\documentclass[a4paper,12pt]{article}
\usepackage{amsmath,amsthm,amssymb}
\usepackage[T2A]{fontenc}
\usepackage[english]{babel}
\usepackage[cp1251]{inputenc}
\usepackage{amscd}

\newskip\prethm \prethm3.0pt plus1.3pt minus.4pt
\newskip\posthm \posthm2.7pt plus1.4pt minus.3pt

\newtheoremstyle{STATEMENT}{\prethm}{\posthm}{\itshape}{\parindent}{\scshape}{.}{.6em plus.2em minus.1em}{}

\newtheoremstyle{EXPLANATION}{\prethm}{\posthm}{}{\parindent}{\scshape}{.}{.6em plus.2em minus.1em}{}

\theoremstyle{STATEMENT}
\newtheorem{theorem}{Theorem}

\newtheorem{assertion}{Assertion}

\theoremstyle{EXPLANATION}

\newtheorem{remark}{Remark}

\DeclareMathOperator{\Ker}{Ker}

\title{Elementary proof of Jordan-Kronecker theorem}
\author{Ivan K. Kozlov}
\date{}

\begin{document}

\maketitle

\begin{abstract}
In this paper we prove the Jordan-Kronecker theorem which gives a
canonical form for a pair of skew-symmetric bilinear forms on a
finite-dimensional vector space over an algebraically closed field.
\end{abstract}

\section{Introduction}

The Jordan--Kronecker theorem gives a canonical form for a pair of
skew-symmetric bilinear form on a finite-dimensional vector space
over an algebraically closed field. This theorem from linear algebra
has recently found various applications in various fields of
mathematics (see, for example \cite{Bolsinov_Oshemkov},
\cite{Gelfand_Zhaharevich}). The proof of the Jordan--Kronecker
theorem can be found in \cite{Gantmacher} and \cite{Thompson}. In
this paper we give a simpler proof of this theorem.

Throughout the paper we assume that all vector spaces are finite
dimensional and the underlying field has characteristic $\ne 2$.

\begin{theorem}[Jordan--Kronecker]
\label{T:Jordan-Kronecker_theorem} Let $A$ and $B$ be skew-symmetric
bilinear forms on a vector space $V$ over a field $\mathbb{K}$. If
the field $\mathbb{K}$ is algebraically closed, then there exists a
basis of the space $V$ such that the matrices of both forms $A$ and
$B$ are block-diagonal matrices:

{\footnotesize
$$
A =
\begin{pmatrix}
A_1 &     &        &      \\
    & A_2 &        &      \\
    &     & \ddots &      \\
    &     &        & A_k  \\
\end{pmatrix}
\quad  B=
\begin{pmatrix}
B_1 &     &        &      \\
    & B_2 &        &      \\
    &     & \ddots &      \\
    &     &        & B_k  \\
\end{pmatrix}
$$
}

where each pair of corresponding blocks $A_i$ and $B_i$ is one of
the following:

\begin{enumerate}

\item Jordan block with eigenvalue $\lambda \in \mathbb{K}$
{\footnotesize
$$
A_i =\left(
\begin{array}{c|c}
  0 & \begin{matrix}
   \hphantom- \lambda &\hphantom-  1&        & \\
      & \hphantom- \lambda & \ddots &     \\
      &        & \ddots & \hphantom- 1  \\
      &        &        & \hphantom- \lambda   \\
    \end{matrix} \\
  \hline
  \begin{matrix}
  -\lambda  &        &   & \\
  -1   & -\lambda &     &\\
      & \ddots & \ddots &  \\
      &        & -1   & -\lambda \\
  \end{matrix} & 0
 \end{array}
 \right)
\quad  B_i= \left(
\begin{array}{c|c}
  0 & \begin{matrix}
    \hphantom- 1 & &        & \\
      & \hphantom- 1 &  &     \\
      &        & \ddots &   \\
      &        &        & \hphantom- 1   \\
    \end{matrix} \\
  \hline
  \begin{matrix}
  -1  &        &   & \\
     & -1 &     &\\
      &  & \ddots &  \\
      &        &    & -1 \\
  \end{matrix} & 0
 \end{array}
 \right)
$$
}
\item Jordan block with eigenvalue $\infty$
{\footnotesize
$$
A_i = \left(
\begin{array}{c|c}
  0 & \begin{matrix}
   \hphantom- 1 & &        & \\
      &\hphantom- 1 &  &     \\
      &        & \ddots &   \\
      &        &        & \hphantom- 1   \\
    \end{matrix} \\
  \hline
  \begin{matrix}
  -1  &        &   & \\
     & -1 &     &\\
      &  & \ddots &  \\
      &        &    & -1 \\
  \end{matrix} & 0
 \end{array}
 \right)
\quad B_i = \left(
\begin{array}{c|c}
  0 & \begin{matrix}
    0 & \hphantom-  1&        & \\
      & 0 & \ddots &     \\
      &        & \ddots & \hphantom-  1  \\
      &        &        & 0   \\
    \end{matrix} \\
  \hline
  \begin{matrix}
  0  &        &   & \\
  -1   & 0 &     &\\
      & \ddots & \ddots &  \\
      &        & -1   & 0 \\
  \end{matrix} & 0
 \end{array}
 \right)
$$
}
\item Kronecker block
{\footnotesize
$$
A_i = \left(
\begin{array}{c|c}
  0 & \begin{matrix}
   \hphantom-  1 & 0      &        &     \\
      & \ddots & \ddots &     \\
      &        &  \hphantom-  1    &  0  \\
    \end{matrix} \\
  \hline
  \begin{matrix}
  -1  &        &    \\
  0   & \ddots &    \\
      & \ddots & -1 \\
      &        & 0  \\
  \end{matrix} & 0
 \end{array}
 \right) \quad  B_i= \left(
\begin{array}{c|c}
  0 & \begin{matrix}
    0 & \hphantom-  1      &        &     \\
      & \ddots & \ddots &     \\
      &        &   0    & \hphantom-  1  \\
    \end{matrix} \\
  \hline
  \begin{matrix}
  0  &        &    \\
  -1   & \ddots &    \\
      & \ddots & 0 \\
      &        & -1  \\
  \end{matrix} & 0
 \end{array}
 \right)
$$
}

\end{enumerate}

A Kronecker block is a $(2k+1) \times (2k+1)$ block, where $k \geq
0$. In particular, if $k=0$, then $A_i$ and $B_i$ are two $1\times
1$ zero matrices
$$
A_i =
\begin{pmatrix}
0
\end{pmatrix} \quad  B_i=
\begin{pmatrix}
0
\end{pmatrix}
$$
\end{theorem}

\begin{remark}
It is easy to prove that the Jordan--Kronecker form of two forms $A$
and $B$ is unique up to the order of blocks.
\end{remark}

\begin{remark}

If $e_1, \dots, e_k, f_0, \dots, f_{k-1}$ is the basis of a Jordan
block with eigenvalue $\infty$, then in the basis $f_0, e_1, f_1,
\dots, f_{k-1}, e_k$ the matrices of the forms are

$$
A_i = \begin{pmatrix}
0 & -1 &   &    &        &   &    \\
1 & 0  &   &    &        &   &    \\
  &    & 0 & -1 &        &   &    \\
  &    & 1 & 0  &        &   &    \\
  &    &   &    & \ddots &   &    \\
  &    &   &    &        & 0 & -1 \\
  &    &   &    &        & 1 & 0  \\
\end{pmatrix}
\quad B_i =\begin{pmatrix}
0 &    &   &         &        &   &    \\
  & 0  & 1 &         &        &   &    \\
  & -1 & 0 &         &        &   &    \\
  &    &   & \ddots  &        &   &    \\
  &    &   &         & 0      & 1 &    \\
  &    &   &         &  -1    & 0 &    \\
  &    &   &         &        &   & 0  \\
\end{pmatrix}
$$

If $e_1, \dots, e_k, f_0, f_1, \dots, f_k$ is the basis of a
Kronecker block, then in the basis $f_0, e_1, f_1, e_2, f_2, \dots,
e_k, f_k$ the matrices of the forms are

$$
A_i = \begin{pmatrix}
 0  & -1 &         &        &   &    \\
 1 & 0 &         &        &   &    \\
    &   & \ddots  &        &   &    \\
    &   &         & 0      & -1 &    \\
    &   &         &  1    & 0 &    \\
    &   &         &        &   & 0
\end{pmatrix}
\quad B_i =\begin{pmatrix}
0 &    &   &         &        &     \\
  & 0  & 1 &         &        &     \\
  & -1 & 0 &         &        &     \\
  &    &   & \ddots  &        &     \\
  &    &   &         & 0      & 1   \\
  &    &   &         &  -1    & 0
\end{pmatrix}
$$

\end{remark}

\section{Self-adjoint operator on a symplectic space.}
First, let us consider the case when one of the forms in
nondegenerate. (Without loss of generality this is form $B$). Let us
restate the problem.

Recall that any bilinear function $B: V \times V \to \mathbb{K}$
defines a map $B: V \to V^{*}$ given by the formula
$$
\langle Bu, v \rangle =B(u, v),
$$
where $\langle Bu, v \rangle$ is the value of a covector $Bu$ on a
vector $v$. If a bilinear form is nondegenerate $\Ker B \ne 0$, then
it defines an isomorphism between the given space and its dual $V
\simeq_{B} V^{*}$.

Put $P = B^{-1} A: V \to V$. The operator $P$ is self-adjoint with
respect to both forms $A$ and $B$. $$ A(Pu, v) = A(u, Pv), \quad
B(Pu, v) = B(u, Pv)
$$

In the sequel we need the following simple assertions.

\begin{assertion}
If $\mathcal P$ is a self-adjoint operator on a symplectic space
$V$, then the orthogonal complement of an invariant subspace $W
\subset V$ is invariant. That is $$ PW \subset W \Rightarrow
PW^{\perp} \subset W^{\perp} $$
\end{assertion}
\begin{proof}
For any $v \in W^{\perp}$ we have $ B(u, Pv) = B(Pu, v) = 0, $ since
$Pu \in W$.
\end{proof}

\begin{assertion} \label{A:IsotrSubs_for_SelfAdj_on_Symp}
Let $P$ be a self-adjoint operator on a symplectic space $(V,
\omega)$. Then for any vector $v \in (V, \omega)$ all vectors $v,
Pv, \dots, P^n v, \dots$ are pairwise orthogonal.
\end{assertion}
\begin{proof}
Evidently, $B(P^{i}v, P^{j}v) = B(P^{i+j}v, v) = B(v, P^{i+j}v) =
0$.
\end{proof}

Now Theorem \ref{T:Jordan-Kronecker_theorem} can be restated as
follows.

\begin{theorem} \label{T:SelfAdjOper_On_SympSpace}
For any self-adjoint operator $\mathcal{P}: V^{2n} \to V^{2n}$ on a
symplectic space $(V^{2n}, \mathcal{B})$ over an algebraically
closed field $\mathbb{K}$ there exists a basis of $V^{2n}$ such that
the matrix $P$ of the operator $\mathcal{P}$ and the matrix $B$ of
the form $\mathcal{B}$ are block-diagonal matrices

{\footnotesize
$$
P =
\begin{pmatrix}
P_1 &     &        &      \\
    & P_2 &        &      \\
    &     & \ddots &      \\
    &     &        & P_k  \\
\end{pmatrix}
\quad  B=
\begin{pmatrix}
B_1 &     &        &      \\
    & B_2 &        &      \\
    &     & \ddots &      \\
    &     &        & B_k  \\
\end{pmatrix}
$$
}
 Each pair of blocks $P_i$ and $B_i$ has the form {\footnotesize
$$
P_i =\left(
\begin{array}{c|c}
  0 & \begin{matrix}
    \lambda & \hphantom-  1&        & \\
      & \lambda & \ddots &     \\
      &        & \ddots & \hphantom- 1  \\
      &        &        & \lambda   \\
    \end{matrix} \\
  \hline
  \begin{matrix}
  -\lambda  &        &   & \\
  -1   & -\lambda &     &\\
      & \ddots & \ddots &  \\
      &        & -1   & -\lambda \\
  \end{matrix} & 0
 \end{array}
 \right)
\quad  B_i= \left(
\begin{array}{c|c}
  0 & \begin{matrix}
    \hphantom-  1 & &        & \\
      & \hphantom-  1 &  &     \\
      &        & \ddots &   \\
      &        &        & \hphantom-  1   \\
    \end{matrix} \\
  \hline
  \begin{matrix}
  -1  &        &   & \\
     & -1 &     &\\
      &  & \ddots &  \\
      &        &    & -1 \\
  \end{matrix} & 0
 \end{array}
 \right)
$$
}
\end{theorem}

\begin{proof}[Proof. Theorem \ref{T:SelfAdjOper_On_SympSpace}]

The proof is in two steps.

\begin{itemize}

\item[Step 1.] First, let us prove the statement when $\mathcal P$ is
a nilpotent operator $\mathcal{P}^{n} = 0$.

Let us show how to extract one block. Suppose that the degree of the
operator $\mathcal{P}$ is $m$, that is
$$
P^m =0, \quad P^{m-1} \ne 0
$$

Take an arbitrary vector $e_1 \in V$ such that $P^{m-1}e_1 \ne 0$.
Let $e_i = P^{i-1}e_1$. Then $\langle e_1, \dots, e_n \rangle$ is an
isotropic subspace.

Since the form $B$ is nondegenerate on $V$ there exists a vector
$f_n \in V$ such that $$ B(e_i, f_n) = \delta^i_n$$

The existence of the vector $f_n$ easily follows from the following
simple assertions from linear algebra.

\begin{assertion}
Any isotropic subspace is contained in a Lagrangian subspace.
\end{assertion}
\begin{assertion}
Any basis $e_1, \dots, e_n$ of a Lagrangian subspace $L \subset (V,
\omega)$ can be extended to a symplectic basis $e_i, f_j$ of the
space $(V, \omega)$
$$
\omega(e_i, f_j) = \delta^i_j
$$
\end{assertion}

Put $f_i = P^{n-i}f_n$. Then $e_i, f_j$ is a basis of a Jordan
block. It is easy to see that
$$
B(e_i, e_j) =0, \quad B(f_i, f_j) =0
$$
(this is assertion \ref{A:IsotrSubs_for_SelfAdj_on_Symp}). It is
also easy to see that
$$
B(e_i, f_j) = \delta^i_j
$$

Indeed, $B(e_i, f_j) = B(e_i, P^{n-j}f_n) = B(P^{n-j}e_i, f_n) =
B(e_{n+i-j}, f_n)= \delta^{n+i-j}_n = \delta^i_j$.

It means that vectors $e_i, f_j$ are linearly independent. In the
basis $e_i, f_j$ the restrictions of the form $\mathcal{B}$ and the
operator $\mathcal{P}$ to the space $\langle e_i, f_j \rangle$ have
matrices

$$
P = \left(
\begin{array}{c|c}
  \begin{matrix}
    0 &        &        &   \\
    1 & 0      &        &   \\
      &  \ddots & \ddots &   \\
      &        &  1     & 0 \\
    \end{matrix} &  \\
  \hline
   & \begin{matrix}
  0  & 1      &        &     \\
     & 0      & \ddots &     \\
     &        & \ddots &   1 \\
     &        &        & 0   \\
  \end{matrix}
 \end{array}
 \right)
, \quad Q =
\begin{pmatrix}
  0 & E \\
 -E & 0 \\
\end{pmatrix}
$$

\item[Step 2.] General case. The space $V$ decomposes into the sum of generalized eigenspaces of
the operator $\mathcal{P}$
$$
V = \bigoplus_{\lambda \in \mathbb K}V^{\lambda}
$$

Recall that a generalized eigenspace with eigenvalue $\lambda$
consists of all vectors $v \in V$ such that $(P - \lambda E)^m = 0$
for some natural $m \in \mathbb N$.

\begin{assertion}
Generalized eigenspaces are orthogonal with respect to the form
$\mathcal{B}$
$$
V^{\lambda} \perp_{\mathcal{B}} V^{\mu}, \quad \lambda \ne \mu
$$
\end{assertion}
\begin{proof}
If $\lambda \ne \mu$, then the restriction of the operator $(P -
\lambda E)$ to $V_{\mu}$ is nondegenerate. Hence for any vector $v
\in V_{\mu}$ the vector $w = (P -\lambda E)^{-1}v$ such that $(P -
\lambda E)w =v$ is well-defined.

For any $e_\lambda \in V_{\lambda}, e_{\mu} \in V_{\mu}$ we have
$$
B(e_{\lambda}, e_{\mu}) = B(e_{\lambda}, (P - \lambda E)(P - \lambda
E)^{-1} e_{\mu}) =  B((P - \lambda E)e_{\lambda}, (P - \lambda
E)^{-1} e_{\mu})$$
$$ = \dots = B((P - \lambda E)^{m}e_{\lambda}, (P - \lambda E)^{-m}
e_{\mu}) =0
$$ \end{proof}

To conclude the proof, it remains to apply step $1$ to the
restriction of the operator $P - \lambda E$ to the corresponding
generalized eigenspace for each eigenvalue $\lambda$.

\end{itemize}

Theorem \ref{T:SelfAdjOper_On_SympSpace} is proved. \end{proof}

\section{Proof of Jordan--Kronecker theorem.}
\begin{proof}[Proof. Jordan--Kronecker theorem.]
If the form $B$ is nondegenerate, then everything is proved (see
Theorem \ref{T:SelfAdjOper_On_SympSpace}). Suppose that $\Ker B \ne
0$.

Let us show how to extract one block. That block will be either a
Kronecker block or a Jordan block with eigenvalue $\infty$. We need
to do the following:
\begin{enumerate}
\item To decompose the space into a sum of subspaces orthogonal
w.r.t. $A$ and $B$
$$
V = V_m \oplus W_m, \quad V_m \perp_{A, B} W_m
$$
\item To find a basis $e_i, f_j$ of the space $V_m$ such that
$$
A(e_i, f_j) = \delta^{i}_{j+1}, \quad B(e_i, f_j) = \delta^{i}_{j},
$$
and all other pairs of basic vectors are orthogonal w.r.t. $A$ and
$B$.
\end{enumerate}

We construct a block in several steps. On odd steps we search for
vectors $f_i$ and on even steps we try to find  vectors $e_j$. If we
can not find a vector, then we have found a block.

\begin{enumerate}
\item[Step 1.] Take an arbitrary vector $f_{0} \in \Ker
B$ and any additional subspace $W_1 \subset V$

$$
\langle f_0 \rangle + W_1 = V
$$

\item[Step 2.] Take a vector $e_1$ such that
$$
A(e_1, f_0) = 1
$$

Put $V_2 = \langle e_1, f_0 \rangle$ and $W_2 = V_2 ^{\perp_A}$

\item[Step $2k+1$.] After $2k$ steps we have constructed subspaces $V_{2k},
W_{2k}$ and a basis $f_0, e_1, f_1, \dots, f_{k-1}, e_k$ of the
space $V_{2k}$ such that

\begin{enumerate}
\item All vectors $e_i, f_j$ are orthogonal to the space $W_{2k}$
w.r.t. the form $A$ $$V_{2k} \perp_{A} W_{2k}$$

\item All vectors except maybe for $e_k$ are orthogonal to the
space $W_{2k}$ w.r.t. the form $B$ $$V_{2k-1} \perp_{B} W_{2k} $$

\item In the basis $f_0, e_1, f_1, \dots, f_{k-1}, e_k$ the restrictions of forms have matrices

$$
A|_{V_{2k}} = \begin{pmatrix}
0 & -1 &   &    &        &   &    \\
1 & 0  &   &    &        &   &    \\
  &    & 0 & -1 &        &   &    \\
  &    & 1 & 0  &        &   &    \\
  &    &   &    & \ddots &   &    \\
  &    &   &    &        & 0 & -1 \\
  &    &   &    &        & 1 & 0  \\
\end{pmatrix}
\quad B|_{V_{2k}} =\begin{pmatrix}
0 &    &   &         &        &   &    \\
  & 0  & 1 &         &        &   &    \\
  & -1 & 0 &         &        &   &    \\
  &    &   & \ddots  &        &   &    \\
  &    &   &         & 0      & 1 &    \\
  &    &   &         &  -1    & 0 &    \\
  &    &   &         &        &   & 0  \\
\end{pmatrix}
$$

\end{enumerate}

Take a vector $f_k \in W_{2k}$ such that
$$
B(e_k, f_k) = 1.
$$

Put $V_{2k+1} = V_{2k} \oplus \langle f_k \rangle$ and $W_{2k+1} =
V_{2k+1}^{\perp_B} \cap W_{2k}$

\item[Step $2k+2$.] After previous steps we have found subspaces $V_{2k+1},
W_{2k+1}$ and a basis $f_0, e_1, f_1, \dots, e_k, f_k$ of $V_{2k+1}$
such that

\begin{enumerate}
\item All vectors $e_i, f_j$ are orthogonal to the subspace
$W_{2k+1}$ w.r.t. the form $B$ $$V_{2k+1} \perp_{B} W_{2k+1}$$

\item All vector except maybe for $f_k$ are orthogonal to the subspace $W_{2k+1}$ w.r.t. the form $A$ $$V_{2k} \perp_{A} W_{2k+1} $$

\item In the basis $f_0, e_1, f_1, \dots, e_k, f_k$ the restrictions of forms have matrices

$$
A_{V_{2k+1}} = \begin{pmatrix}
 0  & -1 &         &        &   &    \\
 1 & 0 &         &        &   &    \\
    &   & \ddots  &        &   &    \\
    &   &         & 0      & -1 &    \\
    &   &         &  1    & 0 &    \\
    &   &         &        &   & 0
\end{pmatrix}
\quad B_{V_{2k+1}} =\begin{pmatrix}
0 &    &   &         &        &     \\
  & 0  & 1 &         &        &     \\
  & -1 & 0 &         &        &     \\
  &    &   & \ddots  &        &     \\
  &    &   &         & 0      & 1   \\
  &    &   &         &  -1    & 0
\end{pmatrix}
$$
\end{enumerate}

Take a vector $e_{k+1} \in W_{2k}$ such that
$$
A(e_{k+1}, f_k) = 1.
$$

Put $V_{2k+2} = V_{2k+1} \oplus \langle e_{k+1} \rangle$ and
$W_{2k+2} = V_{2k+2}^{\perp_A} \cap W_{2k+1}$

\end{enumerate}

If the algorithm stopped on the $2k$-th step (we could not find a
vector $f_{k+1}$), then $V_{2k}$ is a Jordan block with eigenvalue
$\infty$ and if we stopped on the $(2k+1)$-th step (there is no
vector $e_{k+1}$), then $V_{2k+1}$ is a Kronecker
$(2k+1)\times(2k+1)$ block.

Indeed, it is not hard to see that $V_i$ and $W_i$ form two sets of
nested subspaces
$$
V_1 \subset V_2 \subset V_3 \subset \dots
$$
$$
W_1 \supset W_2 \supset W_2 \supset \dots
$$

After $2k$ steps in the basis $f_0, e_1, f_1, \dots, f_{k-1}, e_k$
(and any additional basis of the space $W_{2k}$) the matrix of the
form $A$ is

$$
A = \left(
\begin{array}{c|c}
\begin{matrix}
   0  & -1 &         &        &     \\
  1 & 0 &         &        &     \\
     &   & \ddots  &        &     \\
     &   &         & 0      & -1   \\
     &   &         &  1    & 0
\end{matrix} &0 \\
\hline
0 & A_{2k}\\
\end{array}
\right),
$$

where $A_{2k}$ is the matrix of restriction of the form $A$ to
$W_{2k}$.

Analogously, after $(2k+1)$ steps in basis $f_0, e_1, f_1, \dots,
e_k, f_k$ (and an arbitrary basis of $W_{2k+1}$) the matrix of the
form $B$ is

$$
B = \left(
\begin{array}{c|c}
\begin{matrix}
0 &    &   &         &        &     \\
  & 0  & 1 &         &        &     \\
  & -1 & 0 &         &        &     \\
  &    &   & \ddots  &        &     \\
  &    &   &         & 0      & 1   \\
  &    &   &         &  -1    & 0
\end{matrix} &0 \\
\hline
0 & B_{2k+1}\\
\end{array}
\right)
$$

where $B_{2k+1}$ is the matrix of restriction of the form $B$ to
$W_{2k+1}$.

This concludes the proof. \end{proof}

\begin{remark}
Actually there is no need to prove Theorem
\ref{T:SelfAdjOper_On_SympSpace}. If the field $\mathbb{K}$ is
algebraically closed, then there always exists a degenerate linear
combination of forms $A + \lambda B$ for some $\lambda \in
\mathbb{K} \cup \{\infty\}$. If the field $\mathbb{K}$ is not
algebraically closed, then Theorem \ref{T:SelfAdjOper_On_SympSpace}
remains true if and only if all eigenvalues of $\mathcal{P}$ lie in
$\mathbb{K}$, or equivalently if the characteristic polynomial of
the operator $\mathcal{P}$ splits into linear factors over
$\mathbb{K}$ (compare to the Jordan normal form theorem). Kronecker
blocks and Jordan blocks with eigenvalue $\infty$ can be extracted
over any field with characteristic $\neq 2$ (we did not use
algebraic closeness of the field in that part of the proof).

\end{remark}

\end{document}